\documentclass[11pt]{article}
\usepackage{amsmath,amsthm,amssymb}
\usepackage[dvips]{graphicx}

\newcommand{\C}{\mathbb C}
\newcommand{\N}{\mathbb N}
\newcommand{\R}{\mathbb R}
\renewcommand{\P}{\mathbb P}
\newcommand{\CAP}{\operatorname{cap}}
\newcommand{\re}{\operatorname{Re}}

\begin{document}


\title{A Lower Bound for the Minimum Deviation of the Chebyshev Polynomial on a Compact Real Set\footnote{published in: East Journal on Approximations {\bf 14} (2008), 223--233.}}
\author{Klaus Schiefermayr\footnote{University of Applied Sciences Upper Austria, School of Engineering and Environmental Sciences, Stelzhamerstrasse\,23, 4600 Wels, Austria, \textsc{klaus.schiefermayr@fh-wels.at}}}
\date{}
\maketitle

\theoremstyle{plain}
\newtheorem{theorem}{Theorem}
\newtheorem{lemma}{Lemma}
\newtheorem{definition}{Definition}
\theoremstyle{definition}
\newtheorem{remark}{Remark}
\newtheorem*{example}{Example}


\begin{abstract}
In this paper, we give a sharp lower bound for the minimum deviation of the Chebyshev polynomial on a compact subset of the real line in terms of the corresponding logarithmic capacity. Especially if the set is the union of several real intervals, together with a lower bound for the logarithmic capacity derived recently by A.Yu.\,Solynin, one has a lower bound for the minimum deviation in terms of elementary functions of the endpoints of the intervals. In addition, analogous results for compact subsets of the unit circle are given.
\end{abstract}



\section{Introduction}


Let $C\subset\C$ be a compact infinite set in the complex plane and let $\P_n$ be the set of all polynomials with complex coefficients of degree $n$. Then $L_n(C)$, defined by
\begin{equation}
L_n(C)\equiv\max_{z\in{C}}|z^n+\sum_{k=0}^{n-1}c_{k}^{*}z^{k}|:=
\min_{c_{k}\in\C}\max_{z\in{C}}|z^n+\sum_{k=0}^{n-1}c_{k}z^{k}|,
\end{equation}
is usually called the \emph{minimum deviation} of degree $n$ on $C$. The polynomial
\[
M_{n}(z):=z^n+\sum_{k=0}^{n-1}c_{k}^{*}z^{k}\in\P_{n},
\]
for which the minimum is attained, is called the \emph{minimal polynomial} (or \emph{Chebyshev polynomial}) of degree $n$ on $C$. Fekete\,\cite{Fekete} proved that the limit
\begin{equation}
\CAP{C}:=\lim_{n\to\infty}\root{n}\of{L_n(C)}
\end{equation}
exists and the quantity $\CAP{C}$ is called the \emph{logarithmic capacity} (or \emph{Chebyshev constant} or \emph{transfinite diameter}). The logarithmic capacity is monotone, i.e.,
\begin{equation}\label{MonCap}
C_{1}\subseteq{C}_{2}\Longrightarrow\CAP{C_{1}}\leq\CAP{C_{2}}.
\end{equation}
Concerning many other properties of the logarithmic capacity and the connection to potential theory, we refer to \cite{Kirsch} and the references therein. For the logarithmic capacity $\CAP{C}$ and the minimum deviation $L_n(C)$, the inequality
\begin{equation}\label{IneqLnCap1}
L_n(C)\geq(\CAP{C})^n
\end{equation}
may be found in many textbooks and papers like \cite[Appendix\,B]{Simon} and goes back to Szeg{\"o}\,\cite{Szego} and Fekete\,\cite{Fekete}. Inequality \eqref{IneqLnCap1} is sharp: If $C$ is the unit circle then $\CAP{C}=1$, $M_n(z)=z^n$ and $L_n(C)=1$. In this paper, we prove that for all compact \emph{real} infinite sets $C\subset\R$
\begin{equation}\label{IneqLnCap2}
L_n(C)\geq2(\CAP{C})^n,
\end{equation}
where the constant $2$ is best possible. In addition, a large class of compact sets $C\subset\R$ is given such that equality is attained in \eqref{IneqLnCap2}. This is done in Section\,2. Some analogous results concerning compact infinite sets on the unit circle, lying symmetrically with respect to the real line, are given in Section\,3.

\section{Compact Subsets of the Real Line}


The first result gives an identity between the minimum deviation $L_{n}(A)$ and the logarithmic capacity $\CAP{A}$ in the case that the set $A$ is the inverse polynomial image of $[-1,1]$. As usual, for a polynomial $P_{n}\in\P_{n}$,
\begin{equation}
P_{n}^{-1}([-1,1]):=\bigl\{z\in\C:P_{n}(z)\in[-1,1]\bigr\}
\end{equation}
denotes the inverse polynomial image of $[-1,1]$ for the polynomial mapping $P_{n}$.


\begin{theorem}\label{Thm-LnCap}
Let $P_{n}\in\P_{n}$ and $A:=P_{n}^{-1}([-1,1])$. Then
\begin{equation}\label{LnCap}
L_{n}(A)=2\,(\CAP{A})^n.
\end{equation}
\end{theorem}
\begin{proof}
Let $P_{n}(z)=c_{n}z^{n}+\ldots\in\P_{n}$, $c_{n}\in\C\setminus\{0\}$, and let $T_{k}(z)=\cos(k\arccos(z))$ be the classical Chebyshev polynomial. In \cite{Peh-1996}, Peherstorfer proved that
\[
\frac{2}{(2c_{n})^{k}}\,T_{k}(P_{n}(z))=z^{kn}+\ldots,\quad~k=1,2,3,\ldots,
\]
is a sequence of minimal polynomials of degree $kn$ on $A$ with minimum deviation
\[
L_{kn}(A)=\frac{2}{(2|c_{n}|)^{k}}
\]
thus
\[
\CAP{A}=\lim_{k\to\infty}\sqrt[kn]{L_{kn}(A)}
=\lim_{k\to\infty}\sqrt[kn]{\frac{2}{(2|c_{n}|)^{k}}}
=\frac{1}{(2|c_{n}|)^{1/n}}=\frac{1}{\sqrt[n]{2}}\,\sqrt[n]{L_{n}(A)}.
\]
\end{proof}


\begin{example}\hfill{}
\begin{enumerate}
\item For $I:=[-1,1]$, we have $T_{n}^{-1}([-1,1])=I$, $L_{n}(I)=\frac{1}{2^{n-1}}$ and $\CAP{I}=\frac{1}{2}$, thus, for the set $I$, relation\,\eqref{LnCap} holds.
\item For $0<\alpha<1$ and $n$ even, the minimal polynomial on $E_{\alpha}:=[-1,-\alpha]\cup[\alpha,1]$ is
\[
M_{n}(z)=\frac{1}{2^{n-1}}(1-\alpha^{2})^{\frac{n}{2}}\,
T_{\frac{n}{2}}\bigl(\frac{2z^2-\alpha^2-1}{1-\alpha^2}\bigr)
\]
which gives
\[
L_{n}(E_{\alpha})=\frac{1}{2^{n-1}}(1-\alpha^{2})^{\frac{n}{2}}\quad\text{and}\quad
\CAP{E_{\alpha}}=\tfrac{1}{2}\sqrt{1-\alpha^{2}},
\]
thus, for the set $E_{\alpha}$, relation\,\eqref{LnCap} holds if $n$ is even.
\item Figure\,\ref{Fig_TPolynomial} shows a polynomial $P_{n}$ of degree $n=9$, for which the inverse polynomial image $A:=P_{n}^{-1}([-1,1])$ consists of four real intervals.
\end{enumerate}
\end{example}


\begin{figure}[ht]
\begin{center}
\includegraphics[scale=1.0]{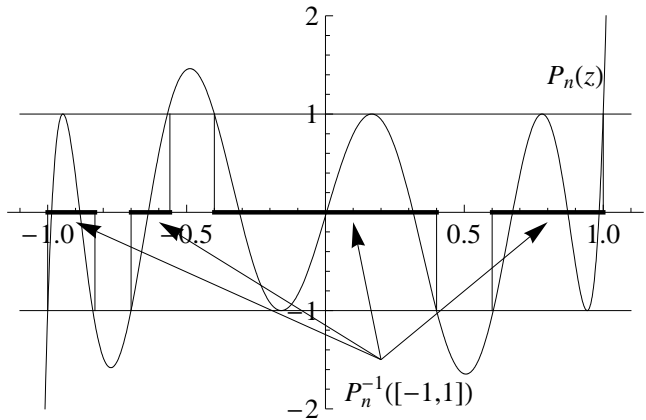}
\caption{\label{Fig_TPolynomial} Polynomial $P_{n}$ of degree $n=9$, for which the inverse polynomial image consists of four real intervals}
\end{center}
\end{figure}


\begin{lemma}\label{Lemma-MinPol}
Let $n\in\N$, let $C\subset\R$ be compact and infinite, let $M_{n}\in\P_{n}$ be the minimal polynomial of degree $n$ on $C$ with minimum deviation $L_{n}\equiv{L}_{n}(C)$, and let $\tilde{M}_{n}:=M_{n}/L_{n}$. Then $C':=\tilde{M}_{n}^{-1}([-1,1])$ is the union of $\ell'$ finite disjoint real intervals, where $1\leq\ell'\leq{n}$, and
\begin{equation}\label{Ca-B}
C\subseteq{C'}\subset\R.
\end{equation}
Moreover, $M_{n}$ is the minimal polynomial of degree $n$ on every compact set $C''$, for which $C\subseteq{C}''\subseteq{C'}$ holds.
\end{lemma}
\begin{proof}
Since $C$ is a real set, the minimal polynomial $M_n$ has only real coeffcients. By the alternation theorem, there are $n+1$ alternation points $x_0,x_1,\ldots,x_n\in{C}$, $x_0<x_1<\ldots<x_n$, with $M_n(x_j)=(-1)^{n-j}L_n$, $j=0,1,\ldots,n$. Thus $M_n$ has $n$ simple zeros $\xi_1<\xi_2<\ldots<\xi_n$ and
\[
x_0<\xi_1<x_1<\xi_2<x_2<\ldots<x_{n-1}<\xi_n<x_n.
\]
Thus, in each interval $(\xi_j,\xi_{j+1})$, $j=1,2,\ldots,n-1$, there is exactly one point $y_j\in(\xi_j,\xi_{j+1})$ with $M_n'(y_j)=0$ and $|M_n(y_j)|\geq{L}_n$, which gives that $\tilde{M}_{n}^{-1}([-1,1])$ is a subset of $\R$ and consists of a finite number (at most $n$) real intervals. The last assertion is an immediate consequence of the alternation theorem.
\end{proof}


With the help of Theorem\,\ref{Thm-LnCap} and Lemma\,\ref{Lemma-MinPol}, we are able to establish the main result.


\begin{theorem}\label{Thm-LB}
Let $C\subset\R$ be compact and infinite, then, for each $n\in\N$,
\begin{equation}\label{IneqLnCap3}
L_{n}(C)\geq2\,(\CAP{C})^{n},
\end{equation}
where equality is attained if there exists a polynomial $P_{n}\in\P_{n}$ such that\\ $P_{n}^{-1}([-1,1])=C$.
\end{theorem}
\begin{proof}
Let $C'$ be defined as in Lemma\,\ref{Lemma-MinPol}, then
\begin{align*}
L_{n}(C)&=L_{n}(C')\qquad\text{by~Lemma\,\ref{Lemma-MinPol}}\\
&=2\,(\CAP{C'})^{n}\qquad\text{by~Lemma\,\ref{Lemma-MinPol}~and~Theorem\,\ref{Thm-LnCap}}\\
&\geq2\,(\CAP{C})^{n}\qquad\text{by~\eqref{Ca-B}~and~\eqref{MonCap}}
\end{align*}
\end{proof}


\begin{remark}
Note that for each $n\in\N$, the constant $2$ in inequality \eqref{IneqLnCap3} is the \emph{best possible} constant which is independent of the set $C$.
\end{remark}


Of special interest is of course if $C$ is the union of several real intervals. For this reason, let us consider the case of $\ell$ intervals, $\ell\in\{2,3,4,\ldots\}$,
\begin{equation}\label{E}
E:=\bigcup_{j=1}^{\ell}[a_{2j},a_{2j-1}],
\end{equation}
where $a_{1},a_{2},\dots,a_{2\ell}$ are fixed with
\begin{equation}\label{aj}
-1=:a_{2\ell}<a_{2\ell-1}<\ldots<a_{2}<a_{1}:=1.
\end{equation}
A.Yu.\,Solynin\,\cite{Solynin} derived a very accurate lower bound for the logarithmic capacity of $E$:


\begin{theorem}[\cite{Solynin}]\label{Thm-Solynin}
Let $E$ be given in \eqref{E} and let
\begin{equation}\label{phipsi}
\varphi_{j}:=\arccos(a_{2j-1}),\quad\psi_{j}:=\arccos(a_{2j}),\quad j=1,2,\ldots,\ell,
\end{equation}
and let $\gamma_{j}$ and $\delta_{j}$ be characterised by
\begin{gather}
\gamma_{1}=0,\quad\gamma_{\ell}=\pi,\quad\varphi_{j}\leq\gamma_{j}\leq\psi_{j},\quad~j=2,3,\ldots,\ell-1,\label{gamma}\\
\psi_{j-1}\leq\delta_{j}\leq\varphi_{j},\quad~j=2,3,\ldots,\ell.\label{delta}
\end{gather}
Then
\begin{equation}\label{IneqCap1}
\CAP{E}\geq\frac{1}{2}\prod_{j=1}^{\ell-1}\Bigl(\sin\tfrac{(\psi_{j}-\gamma_{j})\pi}{2(\delta_{j+1}-\gamma_{j})}\Bigr)
^{\frac{2(\delta_{j+1}-\gamma_{j})^{2}}{\pi^{2}}}\cdot\Bigl(\sin\tfrac{(\gamma_{j+1}-\varphi_{j+1})\pi}{2(\gamma_{j+1}-\delta_{j+1})}\Bigr)
^{\frac{2(\gamma_{j+1}-\delta_{j+1})^{2}}{\pi^{2}}}.
\end{equation}
\end{theorem}


\begin{remark}\hfill{}
\begin{enumerate}
\item Note that the $\gamma_{j}$ and $\delta_{j}$ may be chosen arbitrarily subject to the inequalities \eqref{gamma} and \eqref{delta}. In order to get the best lower bound, one has to maximize the right hand side of \eqref{IneqCap1} over $\gamma_{j},\delta_{j}$, providing \eqref{gamma} and \eqref{delta}.
\item A possible choice for $\gamma_{j}$ and $\delta_{j}$ is
\begin{equation}
\begin{aligned}
\gamma_{j}&:=\tfrac{1}{2}(\varphi_{j}+\psi_{j}),\qquad j=2,3,\ldots,\ell-1,\\
\delta_{j}&:=\tfrac{1}{2}(\varphi_{j}+\psi_{j-1}),\qquad
j=2,3,\ldots,\ell,
\end{aligned}
\end{equation}
for which equality is attained in \eqref{IneqCap1} if $E$ has a special form (see \cite{Solynin}), and for which inequality \eqref{IneqCap1} reads as
\begin{equation}\label{IneqCap2}
\begin{aligned}
\CAP{E}\geq&\frac{1}{2}\prod_{j=1}^{\ell-1}
\Bigl(\sin\tfrac{(\psi_{j}-\varphi_{j})\pi}{2(\varphi_{j+1}-\varphi_{j})}\Bigr)
^{\frac{(\varphi_{j+1}-\varphi_{j})^{2}}{2\pi^{2}}}
\cdot\prod_{j=1}^{\ell-2}
\Bigl(\sin\tfrac{(\psi_{j+1}-\varphi_{j+1})\pi}{2(\psi_{j+1}-\psi_{j})}\Bigr)
^{\frac{(\psi_{j+1}-\psi_{j})^{2}}{2\pi^{2}}}\\
&\qquad\cdot\Bigl(\sin\tfrac{(\pi-\varphi_{\ell})\pi}{2\pi-\varphi_{\ell}+\psi_{\ell-1}}\Bigr)
^{\frac{(2\pi-\varphi_{\ell}+\psi_{\ell-1})^{2}}{2\pi^{2}}}
\end{aligned}
\end{equation}
\item In the simplest case of $\ell=2$ intervals, the capacity $\CAP{E}$ may be computed with the help of Jacobi's elliptic and theta functions, see \cite{Achieser-1930}. Numerical computations show that the lower bound \eqref{IneqCap1} (for optimal $\delta_{2}$) is very accurate and even the lower bound \eqref{IneqCap2} is very good. An upper bound in terms of elementary functions for the logarithmic capacity of two intervals is given by the author in \cite{Sch-2008}.
\item For the representation of the logarithmic capacity of three intervals with the help of theta functions, see \cite{FallieroSebbar-1999} and \cite{FallieroSebbar-2001}.
\item In \cite{PehSch-1999-1} and \cite{PehSch-1999-2}, polynomials $P_{n}\in\P_{n}$, for which $P_{n}^{-1}([-1,1])$ is the union of several intervals, are characterised by a polynomial system for the extremal points of $P_{n}$, see also \cite{Sch-2007}. With this polynomial system, we numerically computed inverse polynomial images $A$ consisting of $\ell=3,4$ and $5$ intervals, the minimum deviation $L_{n}(A)$ and, by \eqref{LnCap}, the capacity $\CAP{A}$, and compared $\CAP{A}$ with the lower bound \eqref{IneqCap1}. All numerical computations gave a relative error of less than $2\%$, which emphasizes the quality of the lower bound \eqref{IneqCap1}.
\item With inequality \eqref{IneqLnCap3}, together with \eqref{IneqCap1} or \eqref{IneqCap2}, one has an excellent lower bound for the minimum deviation of degree $n$ on $E$ in terms of elementary functions of the endpoints $a_{j}$.
\item A very good review of minimal polynomials on several intervals can be found in \cite{Peh-1997}.
\item Let us consider the set of all $E$ consisting of $\ell$ intervals, which are the inverse polynomial image of $[-1,1]$, i.e., for which there exists a polynomial $P_{n}\in\P_{n}$ such that $P_{n}^{-1}([-1,1])=E$ holds. Peherstorfer\,\cite{Peh-2001}, Totik\,\cite{Totik-2001} and Bogatyr{\"e}v\,\cite{Bogatyrev} proved that this set is \emph{dense} in the set of all $E$ consisting of $\ell$ intervals.
\item Let us take a brief look at how the inequality of Theorem\,\ref{Thm-LB} corresponds to the results of Achieser\,\cite{Achieser-1932},\,\cite{Achieser-1933} in the case of two intervals, i.e., $\ell=2$. For the sequence
\begin{equation}\label{Sequence}
\Bigl(\frac{L_{k}(E)}{(\CAP{E})^k}\Bigr)_{k=1}^{\infty}
\end{equation}
he proved the following:
\begin{enumerate}
\item If there exists a polynomial $P_{n}\in\P_{n}$, such that $P_{n}^{-1}([-1,1])=E$, then the sequence \eqref{Sequence} has a finite number of accumulation points from which the smallest one is equal to $2$.
\item  If there is no polynomial $P_{n}\in\P_{n}$, such that $P_{n}^{-1}([-1,1])=E$, then the accumulation points of the sequence \eqref{Sequence} fill out an entire interval of which the left bound is equal to $2$.
\end{enumerate}
\end{enumerate}
\end{remark}


Deriving an upper bound for $L_n(E)$ is more complicated and has been achieved recently by Totik in \cite{Totik-2009}. Since we give an analogous result for several arcs of the unit circle (see Theorem\,\ref{Thm-UB}), we state here the result.


\begin{theorem}[Totik\,\cite{Totik-2009}]\label{Thm-Totik}
Let $E$ be given in \eqref{E}, then there exists a constant $K$ depending only on $E$ such that for each $n\in\N$
\begin{equation}\label{UpperBound-Totik}
L_{n}(E)\leq{K}(\CAP{E})^{n}.
\end{equation}
\end{theorem}

\section{Compact Subsets of the Unit Circle}


In this section, we give analogous bounds for the minimum deviation of the minimal polynomial on a compact subset of the unit circle, lying symmetrically with respect to the real axis.\\
Let $C\subseteq[-1,1]$ be a compact infinite set and let
\begin{equation}\label{Gamma}
\Gamma:=\bigl\{z\in\C:|z|=1,\re{z}\in{C}\bigr\},
\end{equation}
i.e., the set $C$ is the projection of $\Gamma$ onto the real axis. Consider the mapping $x=\tfrac{1}{2}(z+\tfrac{1}{z})$, $z\in\C$, $|z|=1$, $x\in[-1,1]$, for which
\begin{equation}\label{Tn}
T_{n}(x)=\frac{1}{2}\bigl(z^n+\frac{1}{z^n}\bigr)=\re(z^{n})
\end{equation}
holds for the classical Chebyshev polynomial $T_{n}$. Obviously
\begin{equation}\label{CiffGamma}
x\in{C} \iff z\in\Gamma.
\end{equation}
First, let us give a result of Robinson\,\cite{Robinson} concerning an identity between the capacities of $C$ and $\Gamma$.


\begin{theorem}[Robinson\,\cite{Robinson}]\label{Thm-Robinson}
Let $C$ and $\Gamma$ as above, then
\begin{equation}
\CAP\Gamma=\sqrt{2\,\CAP{C}}.
\end{equation}
\end{theorem}


\begin{theorem}
Let $C$ and $\Gamma$ as above and let $n\in\N$. Let $b_{n}:=1$, let
\[
P_{n}(z):=b_{n}z^n+b_{n-1}z^{n-1}+\ldots+b_1z+b_0
\]
be the minimal polynomial of degree $n$ on $\Gamma$ and let $k^*\in\{0,1,\ldots,n\}$ be defined by $k^*:=\min\{k:b_{k}\neq0\}$. If $k^*\neq{n}$, i.e., if $P_{n}(z)\neq{z}^n$, then
\begin{equation}
L_{n}(\Gamma)\geq\sqrt{2|b_{k^*}|}\,(\CAP\Gamma)^{n-k^*}.
\end{equation}
\end{theorem}
\begin{proof}
Since $\Gamma$ is symmetric with respect to the real axis, the coefficients $b_k$ have to be real. Then the square modulus of $P_n(z)$ may be written in the form ($z\in\Gamma$, $x\in{C}$)
\[
|P_n(z)|^2=A_0+2\sum_{\ell=1}^{n}A_{\ell}T_{\ell}(x)\qquad\text{where}\quad
A_{\ell}=\sum_{k=0}^{n-\ell}b_{k}b_{k+\ell}.
\]
By the definition of $k^*$,
\[
|P_n(z)|^2=A_0+2\sum_{\ell=1}^{n-k^*}A_{\ell}T_{\ell}(x)
\]
and $A_{n-k^*}=b_{n}b_{k^*}=b_{k^*}\neq0$. Thus
\begin{align*}
L_{n}(\Gamma)&=\max_{z\in\Gamma}|P_{n}(z)|\\
&=\max_{x\in{C}}\sqrt{\left|A_0+2\sum_{\ell=1}^{n-k^*}A_{\ell}T_{\ell}(x)\right|}\\
&=\sqrt{\max_{x\in{C}}\left|A_0+2\sum_{\ell=1}^{n-k^*}A_{\ell}T_{\ell}(x)\right|}\\
&\geq\sqrt{2^{n-k^*}|b_{k^*}|L_{n-k^*}(C)}\\
&\geq\sqrt{2^{n-k^*+1}|b_{k^*}|(\CAP{C})^{n-k^*}}\qquad\text{by~Theorem\,\ref{Thm-LB}}\\
&=\sqrt{2|b_{k^*}|}\,(\CAP\Gamma)^{n-k^*}\qquad\text{by~Theorem\,\ref{Thm-Robinson}}
\end{align*}
\end{proof}


\begin{theorem}\label{Thm-UB}
Let $E$ be given in \eqref{E} and define $\Lambda:=\{z\in\C:|z|=1,\re{z}\in{E}\}$, i.e., the set $E$ is the projection of $\Lambda$ onto the real axis. Then there exists a constant $B$ depending only on $\Lambda$ such that for each $n\in\N$
\begin{equation}
L_{n}(\Lambda)\leq{B}\,(\CAP\Lambda)^{n}.
\end{equation}
\end{theorem}
\begin{proof}
For the proof, we consider the cases $n$ even and $n$ odd. First let $n$ be even, say $n=2m$. Let
\[
M_{m}(x):=\sum_{k=0}^{m}b_{k}T_{k}(x)=x^{m}+\ldots,
\]
where $b_0,b_1,\ldots,b_{m-1}\in\R$, $b_m=1/2^{m-1}$, be the minimal polynomial on $E$ with corresponding minimum deviation $L_m(E)$. Further, define
\[
P_{2m}(z):=2^mz^m\sum_{k=0}^{m}\tfrac{b_k}{2}(z^k+\tfrac{1}{z^k})=z^{2m}+\ldots
\]
Then
\begin{align*}
L_{2m}(\Lambda)&\leq\max_{z\in\Lambda}|P_{2m}(z)|\\
&=2^m\max_{z\in\Lambda}|\sum_{k=0}^{m}\tfrac{b_k}{2}(z^k+\tfrac{1}{z^k})|
\qquad\text{by~definition}\\
&=2^m\max_{x\in{E}}|\sum_{k=0}^{m}b_{k}T_{k}(x)|
\qquad\text{by~\eqref{Tn}}\\
&=2^mL_{m}(E)\qquad\text{by~definition}\\
&\leq2^{m}K(\CAP{E})^m\qquad\text{by~Theorem\,\ref{Thm-Totik}}\\
&=K(\CAP\Lambda)^{2m}\qquad\text{by~Theorem\,\ref{Thm-Robinson}}
\end{align*}
Let $n$ be odd, say $n=2m+1$. Let $M_{m}$ be defined as above and define
\[
P_{2m+1}(z):=2^mz^{m+1}\sum_{k=0}^{m}\tfrac{b_k}{2}(z^k+\tfrac{1}{z^k})=z^{2m+1}+\ldots
\]
Then, by the same procedure as in the case $n$ even, we get
\[
L_{2m+1}(\Gamma)\leq{K}(\CAP\Gamma)^{2m}=\tfrac{K}{\CAP\Gamma}\,(\CAP\Gamma)^{2m+1}.
\]
Since $\CAP\Gamma\leq1$, for the constant $B$, we take $B=\frac{K}{\CAP\Gamma}$, and the theorem is proved.
\end{proof}

In conclusion, let us note that the connection of minimal polynomials on several intervals and on the corresponding arcs of the unit circle was investigated in \cite{ThiranDetaille} and \cite{PehSch-2002}.

{\bf Acknowledgment.} The author would like to thank Vilmos Totik and Barry Simon for valuable comments on earlier versions of this paper.


\bibliographystyle{amsplain}

\bibliography{LowerBoundLn}

\end{document}